\documentclass[12pt]{amsart}
\usepackage{amsfonts}
\usepackage{amsthm}
\usepackage{latexsym}
\usepackage{amsmath}
\usepackage{amssymb}
\usepackage{amscd}
\usepackage{amsmath}
\usepackage{mathrsfs}
\usepackage{graphicx}
\usepackage{a4wide}
\usepackage{color}

\usepackage{amssymb}

\usepackage[english,francais]{babel}

\newtheorem{theorem}{Theorem}[section]
\newtheorem{lemma}[theorem]{Lemma}
\newtheorem{problem}[theorem]{Problem}
\newtheorem{e-proposition}[theorem]{Proposition}

\newtheorem{e-definition}[theorem]{Definition\rm}

\newtheorem{example}{\it Example\/}
\newtheorem{theoreme}{Th\'eor\`eme}[section]

\newtheorem{proposition}[theoreme]{Proposition}

\newcommand {\cC} {\mathcal C}
\newcommand {\bC} {\mathbb {C}}

\newcommand {\Ga} {\Gamma}
\newtheorem{THEO}{Theorem}

\newtheorem{CONJ}{Conjecture}

\def\og{\leavevmode\raise.3ex\hbox{$\scriptscriptstyle\langle\!\langle$~}}
\def\fg{\leavevmode\raise.3ex\hbox{~$\!\scriptscriptstyle\,\rangle\!\rangle$}}

\begin{document}

\title{Generalizing Tran's conjecture}

\selectlanguage{english}

\author[R.~B\o gvad]{Rikard B\o gvad}
\address{Stockholm University, Department of Mathematics, SE-106 91
Stockholm, Sweden}
\email {rikard@math.su.se }

\author[I.~Ndikubwayo]{Innocent Ndikubwayo}
\address{Stockholm University, Department of Mathematics, SE-106 91
Stockholm, Sweden}
\email{innocent@math.su.se}

            \author[B.~Shapiro]{Boris Shapiro}


\address{Stockholm University, Department of Mathematics, SE-106 91
Stockholm, Sweden}

\email{shapiro@math.su.se}



\begin{abstract}
A conjecture of Khang Tran \cite{Tr1} claims that for an arbitrary pair of polynomials $A(z)$ and $B(z)$, every zero of every polynomial in the sequence $\{P_n(z)\}_{n=1}^\infty$ satisfying the three-term recurrence relation of length $k$
$$P_n(z)+B(z)P_{n-1}(z)+A(z)P_{n-k}(z)=0
$$
with the standard initial conditions $P_0(z)=1$, $P_{-1}(z)=\dots=P_{-k+1}(z)=0$ which is  not a zero of $A(z)$ 
lies on the real (semi)-algebraic  curve  $\cC\subset \bC$ given by 
$$\Im \left(\frac{B^k(z)}{A(z)}\right)=0\quad {\rm and}\quad 0\le (-1)^k\Re \left(\frac{B^k(z)}{A(z)}\right)\le \frac{k^k}{(k-1)^{k-1}}.$$
In this short note, we  show that for  the recurrence relation (generalizing the latter recurrence of Tran)  given by 
$$P_n(z)+B(z)P_{n-\ell}(z)+A(z)P_{n-k}(z)=0, 
$$
 with coprime $k$ and $\ell$ and the same standard initial conditions as above,  
every  root of $P_n(z)$ which is  not a zero of $A(z)B(z)$ belongs to the real algebraic curve $\cC_{\ell,k}$ given by 
$$\Im \left(\frac{B^k(z)}{A^\ell(z)}\right)=0.$$
\end{abstract}

\maketitle

\section{Basic notions and main result}
\label{sec:int}
Linear recurrence relations with various types of coefficients have been studied for more than a century and  appear in different contexts throughout the whole body of mathematics. 
  A natural univariate set-up in this area is as follows. 

\medskip
Let   $\{P_n(z)\}$ be a polynomial sequence satisfying a finite linear recurrence relation of the form 
\begin{equation}\label{eq:main}
P_n(z)+Q_1(z)P_{n-1}(z)+Q_2(z)P_{n-2}(z)+\dots + Q_k(z)P_{n-k}(z)=0,  
\end{equation}
with polynomials coefficients $Q_1(z), \dots, Q_k(z)$ and  initial polynomial conditions of the form $P_0(z)=p_0(z),\; P_{-1}(z)=p_{-1}(z), \dots,$ and $P_{-k+1}(z)=p_{-k+1}(z).$ 

\begin{problem}\label{prob:stand} In the above notation, describe the (asymptotic) behavior of the roots for polynomials in   $\{P_n(z)\}$. 
\end{problem}   

A major result related to Problem~\ref{prob:stand} has been proven in \cite{BKW1}, \cite{BKW2}. It states that independently of the initial conditions, the sequence  
$\{\mu_n\}$ of the root-counting measures of   $\{P_n(z)\}$ converges in the weak sense to the measure $\mu_{\overline Q}$ supported on $\Ga_{\overline Q}$, where $\overline Q=(Q_1(z), Q_2(z),\dots , Q_k(z))$ and $\Ga_{\overline Q}$ is defined as follows. Consider the symbol equation of (\ref{eq:main}) given by 
\begin{equation}\label{eq:char}
t^k+Q_1(z)t^{k-1}+\dots + Q_k(z)=0.
\end{equation}
For a given fixed $z\in \bC$, let $t_1(z)\ge t_2(z)\ge \dots \ge t_k(z)$ be the $k$-tuple of the absolute values of all  (not necessarily distinct) roots of (\ref {eq:char}) in the non-increasing order. Finally, define $\Ga_{\overline Q}$ as: 
\begin{equation}\label{eq:Ga}
\Ga_{\overline Q}:= \{z\in \bC\;  |\;  t_1(z)=t_2(z)\}.
\end{equation}

The density of $\mu_{\overline Q}$ can be also determined using (\ref{eq:char})  and (\ref{eq:Ga}), but we will not need this expression below.  
 
 In the majority of the situations, the roots of $P_n(z)$ only tend to the Beraha-Kahane-Weiss curve $\Ga_{\overline Q}$ when $n\to \infty$ and it is difficult to say something about their location for finite $n$.   The only general exception from this rule is probably the case when all polynomials in $\{P_n(z)\}$ are real-rooted which is often discussed in the  literature and  important for applications.  
 
 However in \cite{Tr1} Khang Tran was able to find a non-trivial situation in which the roots of $P_n(z)$ are not necessarily real and  (almost all of them) still lie on the limiting curve $\Ga_{\overline Q}$ for all $n$.  Observe that this property is destroyed by  small generic deformations of coefficients/initial polynomials. 
 
 \medskip
 Namely, Conjecture 6 of \cite {Tr1}  claims the following. 

\smallskip
\begin{CONJ}\label{conj:Tran}
For an arbitrary pair of polynomials $A(z)$ and $B(z)$, all  zeros of every polynomial in the sequence $\{P_n(z)\}_{n=1}^\infty$ satisfying the three-term recurrence relation of length $k$
\begin{equation}\label{eq:tran}
P_n(z)+B(z)P_{n-1}(z)+A(z)P_{n-k}(z)=0
\end{equation}
with the standard initial conditions $P_0(z)=1$, $P_{-1}(z)=\dots=P_{-k+1}(z)=0$ which do satisfy $A(z)\neq 0$  
lie on the real (semi)-algebraic  curve  $\cC\subset \bC$ given by 
\begin{equation}\label{eq:trcurve}
\Im \left(\frac{B^k(z)}{A(z)}\right)=0\quad {\rm and}\quad 0\le (-1)^k\Re \left(\frac{B^k(z)}{A(z)}\right)\le \frac{k^k}{(k-1)^{k-1}}.
\end{equation}
Moreover, these roots become dense in $\cC$ when $n\to \infty$.
\end{CONJ}

One can check that  in this specific case, the latter curve $\cC$ given by  (\ref{eq:trcurve})  is exactly the Beraha-Kahane-Weiss curve $\Ga_{\overline Q}$. 
In \cite{Tr1} Conjecture~\ref{conj:Tran} was proven for $k=2, 3, 4$ and in \cite{Tr2} Conjecture~\ref{conj:Tran} was proven for arbitrary $k$, but only for polynomials $P_n(z)$ with sufficiently large $n$. Several other aspects of this problem are discussed in \cite{FoTr}, \cite{TrZu1}, \cite{TrZu2}.  The purpose of this short note is to generalize  and  settle the first part of Conjecture~\ref{conj:Tran}. 

We note that in the case when $k$ and $\ell$  are not coprime, say $k=dk', \ell=d\ell',$ gcd$(k',\ell')=1$ 
and \begin{eqnarray*}\sum_{n=0}^\infty P_n(z)t^n=\frac{1}{1+B(z)t^{\ell'}+A(z)t^{k'}},\end{eqnarray*}
we obtain
\begin{eqnarray*}
\sum_{n=0}^\infty Q_n(z)t^n&=&\frac{1}{1+B(z)t^{\ell}+A(z)t^{k}}\\ &=& \frac{1}{1+B(z)t^{d\ell'}+A(z)t^{dk'}}  \\ &=& \sum_{n=0}^\infty P_n(z)t^{dn}.
\end{eqnarray*}
Thus it suffices to study the zero distribution in the case gcd$(k,\ell)=1$. 

\medskip
\begin{theorem}\label{th:main}
For an arbitrary pair of polynomials $A(z)$ and $B(z)$, all  zeros of every polynomial in the sequence $\{P_n(z)\}_{n=1}^\infty$ satisfying the three-term recurrence relation of length $k$
\begin{equation}\label{eq:general}
P_n(z)+B(z)P_{n-\ell}(z)+A(z)P_{n-k}(z)=0
\end{equation}
with coprime $k$ and $\ell$  and with the standard initial conditions $P_0(z)=1$, $P_{-1}(z)=\dots=P_{-k+1}(z)=0$ which  satisfy the condition $A(z)B(z)\neq 0$  
lie on the real algebraic  curve   given by 
\begin{equation}\label{eq:generalcurve}
\Im \left(\frac{B^k(z)}{A^\ell(z)}\right)=0.
\end{equation}
\end{theorem}

Initial results in this direction together with the inequality determining which part of the curve  given by \eqref{eq:generalcurve} contain the roots of $\{P_n(z)\}$ can be found in a recent paper  \cite{Ndi}   of the second author. 

\medskip
\noindent
{\bf Acknowledgements.} The second author  acknowledges the financial
support provided by Sida Phase-IV bilateral program with Makerere University 2015-2020 under project 316 `Capacity building in Mathematics and its applications'.  The  third author wants to acknowledge the financial support of his research provided by the Swedish Research Council grant  2016-04416.

\section{Proofs}

\begin{lemma}\label{lm:new1} In notation of Theorem~\ref{th:main}, 
\begin{equation}\label{eq:imp}
P_n(z)=\sum_{\substack{i\ge 0, j\ge 0 \\ i\ell+jk=n}}(-1)^{i+j}\binom {i+j}{i} A^j(z) B^i(z).
\end{equation}
\end{lemma} 

\begin{proof} Equation~\eqref{eq:general} together with the standard initial conditions imply that 
\begin{align*}
\sum_{n=0}^\infty P_n(z)t^n&=\frac{1}{1+B(z)t^{\ell}+A(z)t^k}\\&=1 -(B(z)t^\ell+A(z)t^k)+(B(z)t^\ell+A(z)t^k)^2-\cdots
\end{align*}
Comparing the coefficients of $t^n$ on both sides of the above equation gives \eqref{eq:imp}.
\end{proof} 

Let $L:=L_{\ell,k,n}=\{(i,j)\in\mathbb{Z}_{\geqslant 0}^2:i\ell+jk=n\}$. Since the gradient of $x\ell+ky=n$ is negative, the set $L$ is a (possibly empty) finite set, say with $s$ elements. In general, the linear Diophantine equation $x\ell+ky=n$ (with $gcd(k,\ell)=1$) has integer solutions of the form $x=x_0+ku$ and $y=y_0-\ell u$ where $x_0,y_0,u\in \mathbb{Z}$. We can choose an $x_0$ and $y_0$ in a such a way that $L=\{(i_u,j_u):=(x_0+ku,y_0-\ell u)\in\mathbb{Z}_{\geqslant 0}^2:u=1,\ldots,s\}$. For any $1\leq u\leq s$, we have $i_u-i_1=k(u-1)$ and $j_1-j_u=\ell(u-1)$.

Let 
\begin{equation} 
\label{eqn1}G_{\ell,k,n}(\tau):=\sum_{u=1}^{s}\binom{i_u+j_u}{i_u}\tau^{u-1}.
\end{equation}
Then $G_{\ell,k,n}(\tau)$ is the generating function for number of north and east lattice paths from the origin $(0,0)$ to the point $(i_u,j_u)\in L$.  The proof of the next result can be found in \cite{Yu}. See Conjecture $1$ and its proof in Section 2 of the the same paper. (Weaker statements in the same direction can be found in \cite{SW}.)

\begin{THEO}  \label{th:interval} In the above notation, for any given positive integers $\ell,k,n$, $G_{\ell,k,n}(\tau)$ as a polynomial in $\tau$ has only negative roots. \end{THEO}
\begin{example} 
\label{Inn} {\rm Take $\ell=2,\, k=3,\; n=21$. Then $L=\{(0,7),\; (3,5),\;(6,3),\;(9,1)\}$. The generating polynomial is 
\begin{align*}
G_{2,3,21}(\tau)=\binom{0+7}{0}+\binom{3+5}{3}\tau+\binom{6+3}{6}\tau^2+\binom{9+1}{9}\tau^3= 1+56\tau+84\tau^2+10\tau^3.
\end{align*}
The roots of this polynomial are approximately equal to $-7.67175,-0.70989$ and $-0.0183618.$}
\end{example}
\medskip\noindent
\begin{proof}[Proof of Theorem~\ref{th:main}] 

For fixed $k,\ell$ and $n$ with $gcd(k,\ell)=1$ and $k>\ell$, as above set $L=\{(i_u,j_u):=(x_0+ku,y_0-\ell u)\in\mathbb{Z}_{\geqslant 0}^2:u=1,\ldots,s\}$ where $(x_0,y_0)$ is a solution to the equation $i\ell+jk=n$.

 From Lemma \ref{lm:new1}, we have
\begin{align*}
P_{n}(z)&=\sum_{u=1}^s (-1)^{i_u+j_u}\binom{i_u+j_u}{i_u}A^{j_u}(z)B^{i_u}(z) \\
&=  A^{j_1}(z)B^{i_1}(z) (-1)^{i_1+j_1}\sum_{u=1}^s (-1)^{i_u+j_u-(i_1+j_1)}\binom{i_u+j_u}{i_u} A^{j_u-j_1}(z)B^{i_u-i_1}(z) \\
&= (-1)^{i_1+j_1}  A^{j_1}(z)B^{i_1}(z)  \sum_{u=1}^s  (-1)^{(u-1)(k-\ell)}\binom{i_u+j_u}{i_u}\frac{B^{(u-1)k}(z)}{ A^{(u-1)\ell}(z)} \\
&= (-1)^{i_1+j_1}  A^{j_1}(z)B^{i_1}(z)  \sum_{u=1}^s \binom{i_u+j_u}{i_u}\Bigg( (-1)^{(k-\ell)}\frac{B^{k}(z)}{ A^{\ell}(z)}\Bigg)^{u-1}\\
&=\pm B^{i_1}(z)A^{j_1}(z)G_{\ell,k,n}\left((-1)^{(k-\ell)}\frac{B^k(z)}{A^\ell(z)}\right).
\end{align*}
The last equality follows from \eqref{eqn1}. 
If $z_0 \in \mathbb{C}$ is such that $P_n(z_0)=0$ and $A(z_0)B(z_0)\neq 0$, then it follows that
$$G_{\ell,k,n}\left((-1)^{(k-\ell)}\frac{B^k(z_0)}{A^\ell(z_0)}\right)=0.$$ By Theorem \ref{th:interval}, we have 
$$\Im \left((-1)^{(k-\ell)}\frac{B^k(z_0)}{A^\ell(z_0)}\right)=0,$$ which completes the proof.
 \end{proof}

\begin{example} Let  $A(z)=z^3 + z + 1, B(z)=z^2 - 2 z + 7, \ell=2, k=3,$ and $n=21$. Then using Example \ref{Inn}, we get
{\small{\begin{align*}
P_{21}(z)&= -A^7 \, G_{2,3,21} \left(-\frac{B^3}{A^2}\right)\\&=-(z^3+z+1)^7 \Big( 1-56\frac{(z^3+z+1)^3}{(z^2-2z+7)^2}+84 \Big( \frac{(z^3+z+1)^3}{(z^2-2z+7)^2} \Big)^2 - 10 \Big( \frac{(z^3+z+1)^3}{(z^2-2z+7)^2} \Big)^3 \Big).
\end{align*}}}
On simplification, we have
\begin{align*}
P_{21}(z)&=393672761 - 646754633 z + 667797557 z^2 + 98239806 z^3 - 
 1206661925 z^4 + \\& 2171467228 z^5 - 2529964192 z^6 + 2246607369 z^7 - 
 1625784860 z^8 + 969712412 z^9 - \\&486724329 z^{10} + 201422869 z^{11} - 
 68243275 z^{12} + 17375116 z^{13} - 2717833 z^{14} -\\& 196756 z^{15} + 
 295748 z^{16} - 114667 z^{17} + 27963 z^{18} - 4619 z^{19} + 492 z^{20} - 
 19 z^{21}.
\end{align*}
This polynomial $P_{21}(z)$ is the same as one given by \eqref{eq:general}  using Mathematica for the same choice of parameters.
\end{example}

In Figure 1 below, we illustrate Theorem \ref{th:main}  for the above generated  polynomial $P_{21}(z)$.

\begin{figure}[h]\begin{center}
$ 
\begin{array}{c}
\includegraphics[height=10cm, width=11cm]{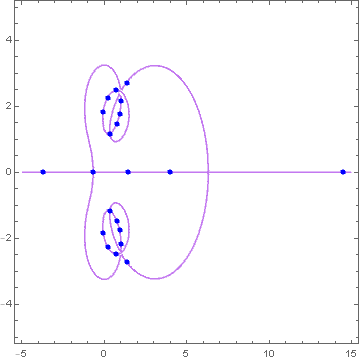}
\end{array}
$
\end{center}
\caption{The graph of  $\Im\left(\frac{B^3(z)}{A^2(z)}\right)=0$ and the zeros of $P_{21}(z)$ on the same axis. }
 \label{a1}

\end{figure}

\section{Final remarks}

\noindent
{\bf 1.} The most important question related to this note is to find the  analog of the inequality in \eqref{eq:trcurve} describing on which part of the real algebraic curve given by \eqref{eq:generalcurve}
the  roots of $\{P_n(z)\}$ are located and become dense when $n\to \infty$. Some special cases are considered in \cite{Ndi}.  

\smallskip
\noindent
{\bf 2.} Choosing an initial $k$-tuple $IN = \{P_0(z),\dots, P_{-k+1}(z)\}$, find other examples/classes of  pairs $(\overline Q, IN)$ where (almost all) zeros of $\{P_n(z)\}$ lie on a fixed curve in the complex plane which is different from an affine line.  One can try to find such examples for relations of order $4$.

\smallskip
\noindent
{\bf 3.} Theorem~\ref{th:interval} apparently has a multivariate generalization when one considers a multivariate polynomial generating function for multinomial coefficients whose indices belong to a hyperplane given by an equation $x_1\ell_1+x_2\ell_2+\dots +x_u\ell_u=n$. It is very tempting to find this generalization and check whether it leads to further implications related to  Problem~\ref{prob:stand} and its version $\grave{a}$ la Tran.


\begin{thebibliography}{00}



\bibitem{BKW1}
S.~Beraha, J.~Kahane, N.~J.~Weiss, {\em Limits of zeroes of recursively
defined polynomials}, Proc. Nat. Acad. USA {\bf 72} (1975), 4209.

\bibitem{BKW2}
S.~Beraha, J.~Kahane, N.~J.~Weiss, {\em Limits of zeros of recursively
defined families of polynomials}, in ``Studies in Foundations and
Combinatorics'', pp. 213--232, Advances in Mathematics Supplementary Studies
Vol. 1, ed. G.-C.~Rota, Academic Press, New York, 1978.



\bibitem{FoTr} T.~Forgacz,  K.~Tran, {\em Hyperbolic polynomials and linear-type generating functions},  arXiv:1810.0152.

\bibitem{Ndi} I.~Ndikubwayo, {\em  Around a conjecture of K.~Tran}, arXiv:1910.00278,  Electronic Journal of Mathematical Analysis and Applications (EJMAA). Vol. 8(2) July. (2020) 16--37.


\bibitem{SW} X.-T.~Su, Y.~Wang, {\em On unimodality problems in Pascal's triangle.}  
The Electronic Journal of Combinatorics (2008)
Volume: 15, Issue: 1, page Research Paper R113, 12 p.-Research Paper R113, 12 p.

\bibitem{Tr1} K.~Tran,  {\em Connections between discriminants and the root distribution of polynomials with rational generating function}, J. Math. Anal. Appl. 410 (2014) 330--340.


\bibitem{Tr2} K.~Tran, {\em The root distribution of polynomials with a three-term recurrence}, J. Math. Anal. Appl. 421 (2015) 878--892.

\bibitem{TrZu1} K.~Tran, A. Zumba, {\em Zeros of polynomials with four-term recurrence}. Involve, a Journal of Mathematics 11(2018), No. 3, 501--518.

\bibitem{TrZu2}  K.~Tran, A.~Zumba, {\em Zeros of polynomials with four-term recurrence and linear coefficients}, arXiv:1808.07133. 


\bibitem{Yu} Y.~Yu, {\em Confirming two conjectures of Su and Wang on binomial coefficients}, Advances in Applied Mathematics 43 (2009) 317--322. 

\end{thebibliography}
\end{document}